\documentclass[12pt]{amsart}

\usepackage{amssymb}
\usepackage{amsfonts}
\usepackage{thmtools, thm-restate}
\usepackage{egothic}
\usepackage[T1]{fontenc}
\usepackage[dvipsnames]{xcolor}
\usepackage{actuarialangle}

\RequirePackage{times}
\RequirePackage{graphicx}
\RequirePackage{hyperref}
\RequirePackage[hang,small,bf]{caption}

\DeclareFontFamily{U}{rcjhbltx}{}
\DeclareFontShape{U}{rcjhbltx}{m}{n}{<->rcjhbltx}{}
\DeclareSymbolFont{hebrewletters}{U}{rcjhbltx}{m}{n}

\let\aleph\relax\let\beth\relax
\let\gimel\relax\let\daleth\relax

\DeclareMathSymbol{\aleph}{\mathord}{hebrewletters}{39}
\DeclareMathSymbol{\beth}{\mathord}{hebrewletters}{98}
\DeclareMathSymbol{\gimel}{\mathord}{hebrewletters}{103}
\DeclareMathSymbol{\daleth}{\mathord}{hebrewletters}{100}\let\dalet\daleth

\DeclareMathSymbol{\lamed}{\mathord}{hebrewletters}{108}
\DeclareMathSymbol{\mem}{\mathord}{hebrewletters}{109}
\DeclareMathSymbol{\ayin}{\mathord}{hebrewletters}{96}
\DeclareMathSymbol{\tsadi}{\mathord}{hebrewletters}{118}
\DeclareMathSymbol{\qof}{\mathord}{hebrewletters}{114}
\DeclareMathSymbol{\shin}{\mathord}{hebrewletters}{152}

\newcommand{\bal}[1]{\begin{align*}#1\end{align*}}

\newcommand{\bi}{\mathbf{i}} 
\newcommand{\bj}{\mathbf{j}}
\newcommand{\bk}{\mathbf{k}}
\newcommand{\Z}{\mathbb{Z}} 
\newcommand{\R}{\mathbb{R}}
\newcommand{\C}{\mathbb{C}}

\newcommand{\bH}{\mathbb{bH}}

\newcommand{\Ad}{\mathrm{Ad}}

\newcommand{\eps}{\varepsilon}
\newcommand{\scP}[2]{\left\langle #1, #2\right\rangle} 
\newcommand{\ovln}[1]{\overline{#1}}

\renewcommand{\phi}{\varphi}
\renewcommand{\Re}{\mathrm{Re}\,} 

\newcommand{\MM}{\mathcal{M}} 
\newcommand{\KK}{\mathcal{K}} 

\newcommand{\pp}{\mbox{\egothfamily p}}
\newcommand{\ac}[1]{\actuarialangle{#1}}

\declaretheoremstyle[headfont=\large\egothfamily, 
postheadspace=\newline, bodyfont=\itshape, spaceabove=0.5cm,
spacebelow=0.5cm, notebraces=| |, shaded={bgcolor={blue!15!gray!5}}]{myLemmaStyle}

\declaretheoremstyle[headfont=\large\egothfamily, 
postheadspace=\newline, bodyfont=\itshape, spaceabove=0.5cm,
spacebelow=0.5cm, notebraces=| |, shaded={bgcolor={yellow!20}}]{myThmStyle}

\declaretheorem[heading=Satz, parent = section, style= myThmStyle]{theo}
\declaretheorem[heading=Lemma, sibling=theo, style= myLemmaStyle]{lemma}
\declaretheorem[heading=Corollary, sibling=theo, style= myLemmaStyle]{coro}

\theoremstyle{definition}
\newtheorem{defi}[theo]{Definition}

\begin{document}

\begin{center}
{\Large A 3D Ginibre point field}

\bigskip
Vladislav Kargin\footnote{{email:
vladislav.kargin@gmail.com; current address: Department of Mathematics, Binghamton University, Binghamton, 13902-6000, USA}} 

\bigskip
\begin{quote}
We introduce a three-dimensional random point field using the concept of the quaternion determinant. Orthogonal polynomials on the space of pure quaternions are defined, and used to construct a kernel function similar to the Ginibre kernel. We find explicit formulas for the polynomials and the kernel, and calculate their asymptotics in the bulk and at the center of coordinates. 
\end{quote}
\end{center}

\tableofcontents

\section{Introduction}

Informally, a random point field is a configuration of points in a space selected according to a joint probability distribution.  If the joint distribution can be written as a product of marginal distributions for different points, then the points are ``independent'' or not
``interacting'', so that their locations are independent from each other. However, there are many interesting point fields, in which points are not independent.

An important class of such dependent point fields are determinantal fields, in which the joint distribution 
can be written as a determinant of a kernel function (see Dyson (\cite{dyson70}), Macchi (\cite{macchi75}, 
Mehta (\cite{mehta04}), Soshnikov (\cite{soshnikov00b}, Hough et al. (\cite{hkpv09})).

In particular, the eigenvalues of non-Hermitian Gaussian random matrices form a
determinantal field in the complex plane (\cite{ginibre65}) with an interesting complex-valued kernel,

\[
K_n(z,w) = \sum_{k = 0}^{n - 1} \frac{(z\overline{w})^k}{k!}.
\]

This random point field is called the Ginibre ensemble. It is quite remarkable since its kernel is among the small number of kernels which are invariant relative to a large group of complex plane transformations. (See Krishnapur \cite{krishnapur08} ).

The Ginibre kernel is quite natural in the sense that polynomials $z^k/\sqrt{k!}$ are orthonormal with respect to the standard Gaussian measure in the plane. In particular, this kernel can be considered as the kernel of the orthogonal projection from the space of all polynomials on $\C$ to the $n$ - dimensional space of complex valued polynomials of degree no larger than $n - 1$.

Our goal in this paper is to develop a parallel theory for a similar
quaternion-valued kernel,

\begin{equation}
\label{defi_kernel}
K_n(z,w) = \sum_{k = 0}^{n - 1} \frac{P_k(z)\overline{P_k(w)}}{h_k}.
\end{equation}

Here $z$ and $w$ are purely imaginary quaternions, which we identify with points in $\R^3$, and polynomials $P_k(z) \in \bH[z]$  are the monic quaternion polynomials of degree $k$, which are orthogonal with respect to the standard Gaussian measure on $\R^3$.

More precisely, let $\Lambda =\mathbb{R}^{3}.$ We identify elements or $\Lambda $ with
pure quaternions, that is, with elements of $\mathbb{H}$ that have no real
component, $z=x_{1}\mathbf{i}+x_{2}\mathbf{j}+x_{3}\mathbf{k}$. 

As the background measure on $\Lambda $, we take

 \bal{
 d\mu(z) = f(z)\, dm(z),
 }
  where $dm(z)$ is the
Lebesgue measure on $\mathbb{R}^3,$ and 

\begin{equation}
\label{background_density}
f(z) = (2\pi)^{-3/2}
e^{-\vert z\vert^{2}/2}.
\end{equation}

(In particular, $\int_{\mathbb{R}^3}d\mu
\left( z\right) =1.$) 

Next, let $\MM$ be the right $\bH$-module of polynomial functions on 
the space of imaginary quaternions
$\Lambda$. Elements of this module are defined as finite sums of $z^{k}a_{k},$ where $a_{k}$ are arbitrary
quaternions and $z \in \Lambda$,%
\begin{equation*}
P\left( z\right) =\sum_{k=0}^{n}z^{k}a_{k}.
\end{equation*} 

Let a scalar
product  on this space be defined by 
\begin{equation}
\scP{u(z)}{w(z)}=\int_{\Lambda }\overline{u}(z) w(z) \, d\mu(z) .  \label{scalar_product}
\end{equation}%
This function is linear in the second argument 
\bal{
\scP{u(z)}{w(z) \alpha} = \scP{u(z)}{w(z)} \alpha 
}
 for every $\alpha \in \bH$, and it also has the property that  
$\scP{w(z)}{u(z)} = \ovln{\scP{u(z)}{w(z)}}$.

In addition, the scalar product is positive definite:
\bal{
\scP{u(z)}{u(z)} \geq 0 \text{ for all } u(z) \in \MM, 
} 
and equality holds only if $u(z) = 0.$

We say that a family of quaternionic functions on $\Lambda ,$ $\{u_{k}\}$ is
\emph{orthonormal}, if 
 
\begin{equation}
\scP{u_k}{u_l} =
\begin{cases}
1, & \text{ if } k = l, 
\\
0, & \text{ if } k \ne l.  
\end{cases}
\end{equation}

Let $P_k(z)$ denote the system of monic orthogonal polynomials. We can show that this system exists and unique by the Gram-Schmidt orthogonalization process. 

 It turns out that $P_k(z) \ne z^k$, and we will study these polynomials in detail. 
The  constants $h_k$ in formula (\ref{defi_kernel}) ensure that polynomials $P_k(z)/\sqrt{h_k}$ have unit norm. 

The rest of the paper is organized as follows. In Section \ref{sectionOverview} we give an overview of main results. In Section \ref{section_defi} we briefly review the theory of quaternion determinantal fields. 
In Section \ref{section_ortho_poly} we study the orthogonal polynomials $P_k$ and the norm constants $h_k$. Section \ref{section_kernel} is devoted to the study of properties of the kernel. And Section \ref{sectionProblems} concludes with open problems.
 
 %

\section{Brief description of the results}
\label{sectionOverview}
For the monic orthogonal polynomials $P_n(x)$ we obtain a formula that relates them to Hermite polynomials $H_n(x)$ and calculate the norm of $P_n(x)$. 

\begin{restatable*}{theo}{theoHermite}
\label{theoHermite}
Let pure quaternion $x = u s$, where $s = |x|$ and $u = x/|x|$. Then 
\begin{equation}
\label{formula_explicit}
P_{n}(x)=\left\{ 
\begin{array}{cc}
\frac{u^n}{s}H_{n+1}(s), & \text{ if }n\text{ is even,} \\ 
\frac{u^n}{s^2} \left( s H_{n+1}(s)+H_n(s)\right) , & \text{if }n\text{ is odd.}%
\end{array}%
\right.
\end{equation}
\end{restatable*}

(The numbering of theorems refers to the section where they are proved. For example, this theorem is restated and proved in Section 4.)

The norms of these polynomials are given by the following formula,

\begin{equation*}
h_{n}:= \|P_n(x)\|^2 = \left\{ 
\begin{array}{cc}
n!(n+2), & \text{ if }n\text{ is odd,} \\ 
(n+1)!, & \text{if }n\text{ is even.}%
\end{array}%
\right.
\end{equation*}%

The kernel $K_n$ is defined by formula (\ref{defi_kernel}), and it turns out that an analogue of the Christoffel-Darboux formula holds. 

\begin{restatable*}{theo}{theoKernelFormula}
\label{theoKernelFormula}
Let $x = u s$ and $y = v t$, where $s = |x|$ and $t = |y|$. 
Suppose $ s \neq  t$. Then
\begin{align*}
K_n(u s, v t) =  \rho_n(s, t)\frac{1 - u v}{2}  + (-1)^n \delta_n(s, t) \frac{1 + u v}{2},
\end{align*}
where 
\begin{align*}
\rho_n(s, t) &= \frac{Q_n(s) Q_{n+1}(t) - Q_{n+1}(s) Q_n(t)}{h_n (t - s)}, \text{ and }
\\
\delta_n(s, t) &= \frac{Q_n(s) Q_{n+1}(t) + Q_{n+1}(s) Q_n(t)}{h_n (t + s)}. 
\end{align*}
\end{restatable*} 
Here the $Q_n(s) \in \Z[s]$  and determined by the following condition. If $u$ is a unit pure quaternion, then 

\bal{
P_{n}(u s) = u^n Q(s).
}

Next, define
\begin{equation}
\label{defiRho}
\rho(s) := \frac{1}{(2\pi)^2} \frac{\sqrt{1 - s^2}}{s^2} \text{ for } s \in (0, 1).
\end{equation}

\begin{restatable*}{theo}{theoDensity}
\label{theoDensity}
Let $\rho_n(x)$ denote the density of the random field at point $x \in \R^3$ with respect to the Lebesque measure on $\R^3$. Assume that $u$ is a pure unit quaternion and $s \in \R$, $s > 0$. Then,
\bal{ 
\lim_{n \to \infty} 2 \sqrt{n} \rho_n\big(2\sqrt{n} u s\big) =
\begin{cases}
\rho(s), & \text{ if } s < 1, \\
0 & \text{ if } s > 1.
\end{cases}
}
\end{restatable*}

Intuitively the theorem says that (i) with high probability a random field point is contained in the ball of radius $2\sqrt{n}$; (ii) the average density of points in the ball is $1/\sqrt{n}$, and (iii) after rescaling the density has the form given by the function $\rho(s)$.

 Next, let us define the \emph{radial density} 
\bal{
\pp_n(r) = \int_{r S^2} \rho_n(x) \, d\mu(x), 
}
where $r S^2$ is the sphere of radius $r$ and $d\mu(x)$ is the measure on $r S^2$ induced by the Lebesgue measure on $\R^3$.
The radial density summarizes how the points are distributed over the different spherical shells. Inside each shell, the distribution is rotationally invariant.

\begin{restatable*}{coro}{coroRadialDensity}
\label{coroRadialDensity}
We have 
\bal{
\lim_{n \to \infty} \frac{1}{2 \sqrt{n}} \pp_n\big(2\sqrt{n} s\big) =
\begin{cases}
\pp(s) :=\frac{1}{\pi}\sqrt{1 - s^2}, & \text{ if } s < 1, \\
0 & \text{ if } s > 1,
\end{cases}
}
\end{restatable*}

Next, we check that if $x$ and $y$ are fixed and different, then the limit of the expression for $K_n(\sqrt{n}x, \sqrt{n}y)$  
 is $0$ for $n\to \infty$ . There are no 2-correlations at this scale. 
 
 In order to understand what happens when  $x$ is sufficiently close to $y$, we consider the case, in which
 $y = x + \frac{\tau}{n}$. We will correct it a bit since the scaling depends on the density of the field, and the density varies with $x$.

 Let 
 \begin{restatable}{equation}{defiXPhi}
 x_0 \in [\eps, 1 - \eps], \text{ and } \phi = \arccos x_0.
\end{restatable}

Define 
\begin{restatable}{equation}{defiSnTn}
\begin{split}
s_n &= 2\sqrt{n + 3/2} \cos\Big(\phi + \frac{\sigma}{2 \sin^2 \phi} \frac{1}{n}\Big), \\
 t_n &= 2 \sqrt{n + 3/2}\cos\Big(\phi + \frac{\tau}{2 \sin^2 \phi} \frac{1}{n}\Big),
\end{split}
\end{restatable}
where $\sigma$ and $\tau$ are real parameters. (We use factor $\sqrt{n + 3/2}$ for the convenience of calculations in the proof. The factor   $\sqrt{n}$ works equally well.) Note that $|s_n - t_n|$ is of the order $1/\sqrt{n}$. 

It will be convenient to use the notation 
\begin{restatable}{equation}{defiKK}
\KK_n\big((u, \sigma), (v, \tau); x_0\big) = 
\frac{2 \sqrt{n + \frac{3}{2}}}{\rho(x_0)}K_n(u s_n, v t_n) 
\sqrt{f(s_n)f(t_n)},
\end{restatable}
where $u$ and $v$ are pure unit quaternions, $\sigma$ and $\tau$ are real, $x_0 \in [\eps, 1 - \eps]$, and $\rho(x_0)$ is as defined in (\ref{defiRho}).

\begin{restatable*}{theo}{theoBulkAsymptotic}
\label{theoBulkAsymptotic}
Let $\sigma, \tau \in \R$ and $u, v \in S^2$, the space of pure unit quaternions. Then,
\bal{
\lim_{n \to \infty} \KK_n\big((u, \sigma), (v, \tau); x_0\big) = \KK\big((u, \sigma), (v, \tau)) :=\frac{\sin(\tau - \sigma)}{\tau - \sigma} \frac{1 - uv}{2}. 
 }
\end{restatable*}

This gives us a kernel on the space $S^2 \times \R$ as a scaling limit of the quaternion kernels on $\R^3$.

At the center the situation is different.

\begin{restatable*}{theo}{theoKernelCenter}
\label{theoKernelFormula}
Let $x = u s$ and $y = v t$, where $s = |x|$ and $t = |y|$. 
Suppose $ s \neq  t$. Then
\begin{align*}
K_n(u s, v t) e^{-\frac{s^2 + t^2}{4}} 
&= \sqrt{\frac{2}{\pi}} \frac{1}{s t}\Big\{\frac{\sin\big[\sqrt{n}(t - s)\big]}{t - s}\frac{1 - u v}{2}  
\\
&+ \frac{\sin\big[\sqrt{n}(t + s)\big]}{t + s } \frac{1 + u v}{2} + O(n^{-1/2})\Big\},
\end{align*}
\end{restatable*}

In the next section we begin a more detailed discussion of these results.
%

\section{Definition of the determinantal point field}
\label{section_defi}
\subsection{Existence results}
 Let $\Lambda =\mathbb{R}^{n}$ and the background measure $\mu $ be a Radon measure on $\Lambda $
(that is, a Borel measure which is finite on compact sets). 
\begin{defi}
A \emph{random point field} $\mathcal{X}$ on $%
\Lambda $ is a \textrm{random} \textrm{integer-valued} positive Radon 
\textrm{measure} on $\Lambda $. 
\end{defi}
For a set $D \in \Lambda$, we can interpret $\mathcal{X}(D)$  as the number of points in $D$ (counted with multiplicities, perhaps). We will assume our fields to be \emph{simple}, so that each point has a multiplicity at most $1$.

A random point field can be described by its correlation functions.

\begin{defi}
A locally integrable function $R_{k}$: $\Lambda ^{k}\rightarrow \mathbb{R}%
_{+}^{1}$ is called a $k$\emph{-point correlation function} of a random
point field $\mathcal{X}$ if for any disjoint measurable subsets $%
A_{1},\ldots ,A_{m}$ of $\Lambda $ and any non-negative integers $%
k_{1},\ldots ,k_{m},$ such that $\sum_{i=1}^{m}k_{i}=k,$ the following
formula holds: 
\begin{eqnarray}
&&\mathbb{E}\prod_{i=1}^{m}\left[ \mathcal{X}\left( A_{i}\right) \ldots
\left( \mathcal{X}\left( A_{i}\right) -k_{i}+1\right) \right] \\
&=&\int_{A_{1}^{k_{1}}\times \ldots \times A_{m}^{k_{m}}}R_{k}\left(
x_{1},\ldots ,x_{k}\right) d\mu \left( x_{1}\right) \cdots d\mu \left(
x_{k}\right) ,  \label{def_correlation_formulas}
\end{eqnarray}%
where $\mathbb{E}$ denote expectation with respect to random measure $%
\mathcal{X}$.
\end{defi}

If we set all $k_i = 1$ and let $A_i$ to be infinitesimally small, then we can see that $R_{k}(x_1,\ldots ,x_k)$ is the density of the probability to find a point in a neighborhood of each of the points $x_i$.

More generally, left hand side can be interpreted as a joint ordered moment of the number of
points in sets $A_{i}.$

A random point field is traditionally called determinantal if all of its
correlation functions can be represented as determinants:%
\begin{equation}
R_{k}(x_{1},\ldots ,x_{k})=\left. \text{\textrm{Det}}[K(x_{i},x_{j})]\right%
\vert _{1\leq i,j\leq k},\text{ for all }k\in \mathbb{N},
\label{correlations_standard}
\end{equation}%
where $K\left( x,y\right) $ is a kernel function that maps $\Lambda \times
\Lambda $ to $\mathbb{C}$. We extend this definition to include the case of quaternion determinants. 
\begin{defi}
A random point field on $%
\Lambda $ is \emph{determinantal} if its correlation functions are given by
the formula: 
\begin{equation}
R_{k}(x_{1},\ldots ,x_{k})=\left. \text{\textrm{Det}}_{M}[K(x_{i},x_{j})]%
\right\vert _{1\leq i,j\leq k},\text{ for all }k\in \mathbb{N},
\label{correlations}
\end{equation}%
where $K$ maps $\Lambda \times \Lambda $ to $\mathbb{H}$, the skew field of
(real) quaternions, and $\mathrm{Det}_{M}$ is the Dyson-Moore determinant.
\end{defi}

For the case when kernel $K(x, y)$ takes values in $\mathbb{C}$, this definition agrees
with definition in (\ref{correlations_standard}). For quaternion kernels $K(x,y),$
this definition only makes sense if the kernel $K$ is self-dual, that is, $%
K(w,z)=\overline{K(z,w)}$ . (See \cite{aslaksen96}, \cite{zhang97}, and \cite%
{farenick_pidkowich03} for more information about quaternion matrices and
their determinants.)

A basic question is \textquotedblleft Which kernels lead to a valid
collection of correlation functions?\textquotedblright\ One sufficient
condition is as follows.

%

\begin{theo}[Dyson-Mehta]
\label{theorem_Dyson} Suppose that the kernel $K_N:\Lambda \times \Lambda
\rightarrow \mathbb{H}$ can be written as 
\begin{equation*}
K_N(x,y)=\sum\limits_{k=1}^{N}u_{k}\left( x\right) \overline{u}_{k}\left(
y\right) ,
\end{equation*}%
where $u_{k}\left( x\right) $ is an orthonormal system of quaternion
functions on $\Lambda $. Then there exists a random point field on $\Lambda $
with the correlations given by (\ref{correlations}), and the total number of
points is almost surely equals $N.$
\end{theo}

This result was essentially proved by Dyson (\cite{dyson70}) and generalized  by Mehta (\cite{mehta04}, Theorem 5.1.4). 

%

By Theorem \ref{theorem_Dyson} the kernel 
\begin{equation*}
K_n(z,w) = \sum_{k = 0}^{n - 1} \frac{P_k(z)\overline{P_k(w)}}{h_k}
\end{equation*}
 defines a valid determinantal field provided that $P_k(z)/h_k$ are orthonormal polynomials with respect to the measure $d\mu(z)$. 
 
 The next section is devoted to the existence and properties of these  polynomials. Then we will study the properties of the kernel.

%

\section{Orthogonal polynomials}
\label{section_ortho_poly}

%

\subsection{Scalar Product}

We look for the orthonormal basis of the space of polynomials $\MM$ with respect to the scalar
product 
\begin{equation}
\left\langle u(z),w(z)\right\rangle =\int_{\Lambda }\overline{u}\left(
z\right) w\left( z\right) d\mu \left( z\right) .  \label{scalar_product}
\end{equation}%
We can find this basis by applying the Gram-Schmidt procedure to the
monomial basis $\left. \left\{ z^{k}\right\} \right\vert _{k=0}^{n}.$ For
this purpose we compute the scalar products of monomials.

\begin{theo}
For all non-negative integers $m$ and $n$ 
\begin{equation*}
\left\langle z^{m},z^{n}\right\rangle =\left\{ 
\begin{array}{cc}
\left( -1\right) ^{\frac{n-m}{2}}(m+n+1)!!, & \text{if }n-m\text{ is even, }
\\ 
0 & \text{if }n-m\text{ is odd.}%
\end{array}%
\right.
\end{equation*}
\label{propo_scalar_product}
\end{theo}

The proof of this result can be found in Appendix \ref{sectionPropoScalar}.

\begin{table}[tbp]
\begin{tabular}{c|ccccccc}
m/n & 0 & 1 & 2 & 3 & 4 & 5 & 6 \\ \hline
0 & 1 & 0 & $-3$ & 0 & $5 \cdot 3$ & 0 & $-7!!$ \\ 
1 & 0 & $3$ & 0 & $-5 \cdot 3$ & 0 & $7!!$ & 0 \\ 
2 & $-3$ & 0 & $5 \cdot 3$ & 0 & $- 7!!$ & 0 & $9!!$ \\ 
3 & 0 & $- 5 \cdot 3$ & 0 & $7!!$ & 0 & $- 9!!$ & 0 \\ 
4 & $5 \cdot 3$ & 0 & $- 7!!$ & 0 & $9!!$ & 0 & $- 11!!$ \\ 
5 & 0 & $7!!$ & 0 & $- 9!!$ & 0 & $11!!$ & 0 \\ 
6 & $- 7!!$ & 0 & $9!!$ & 0 & $- 11!!$ & 0 & $13 !!$ \\ 
&  &  &  &  &  &  & 
\end{tabular}%
\caption{Scalar products of monomials, $\left\langle
z^{m},z^{n}\right\rangle $.}
\label{table_scalar_products}
\end{table}

Since all entries in the matrix of scalar products are real, we derive an
important consequence that the coefficients of  monic orthogonal polynomials
are real.

Note that $\langle z, z \rangle \ne \langle 1, z^2 \rangle$, which means
that the scalar product cannot be written as $\langle P(z), Q(z) \rangle =
\int_{\mathbb{R}}{P(x) Q(x) \mu(dx)} $ for a measure $\mu$ on the real line.

\subsection{Three-term recurrence relation}

By usual means, we can derive the three-term recurrence relation for the
orthogonal polynomials.

\begin{theo}
Suppose that $P_{n}\left( z\right) $ are monic polynomials orthogonal with
respect to the scalar product in (\ref{scalar_product}), and that $%
P_{0}\left( z\right) =1,$ and $P_{1}(z)=z$. Then these polynomials satisfy
the following recurrence relation:%
\begin{equation}
P_{n+1}\left( z\right) =zP_{n}\left( z\right) +\beta _{n}P_{n-1}\left(
z\right) ,  \label{ThreeTermRelation}
\end{equation}%
where $\beta _{n}$ are some real positive coefficients, and 
\begin{equation*}
\beta _{n}=\frac{\langle P_{n},P_{n}\rangle }{\langle P_{n-1},P_{n-1}\rangle 
}.
\end{equation*}
\end{theo}

\textbf{Proof:} Let us start with the expression 
\begin{equation*}
\psi _{n}\left( z\right) =P_{n+1}\left( z\right) -P_{n}\left( z\right)
\left( z-\alpha _{n}\right) -P_{n-1}\left( z\right) \beta _{n}.
\end{equation*}%
Then $\left\langle P_{k},\psi _{n}\right\rangle =0$ for all $k\leq n-2.$ We
also calculate that 
\begin{equation*}
\left\langle P_{n-1},\psi _{n}\right\rangle =-\left\langle
P_{n-1,}zP_{n}\right\rangle -\left\langle P_{n-1},P_{n-1}\right\rangle \beta
_{n}=0,
\end{equation*}%
if we take 
\begin{equation}
\label{equ_beta}
\beta _{n}=-\frac{\left\langle P_{n-1,}zP_{n}\right\rangle }{\left\langle
P_{n-1},P_{n-1}\right\rangle }.
\end{equation}

Similarly, 
\begin{equation*}
\left\langle P_{n},\psi _{n}\right\rangle =-\left\langle
P_{n,}zP_{n}\right\rangle +\left\langle P_{n},P_{n}\right\rangle \alpha
_{n}=0,
\end{equation*}%
if we take 
\begin{equation*}
\alpha _{n}=\frac{\left\langle P_{n,}zP_{n}\right\rangle }{\left\langle
P_{n},P_{n}\right\rangle }.
\end{equation*}

With this choice of $\alpha _{n}$ and $\beta _{n},$ polynomial $\psi _{n}$
has degree $n$ and orthogonal to $P_{k}$ for all $k\leq n$. Since
polynomials $P_{k}$ form a basis, hence $\left\langle \psi _{n},\psi
_{n}\right\rangle =0$. By using the definition of the scalar product (\ref%
{scalar_product}) we conclude that $\psi _{n}=0.$ Hence 
\begin{equation*}
P_{n+1}\left( z\right) =P_{n}\left( z\right) \left( z-\alpha _{n}\right)
+P_{n-1}\left( z\right) \beta _{n}.
\end{equation*}%
By induction it is easy to see that all $P_{2k+1}$ are odd polynomials, all $%
P_{2k}$ are even, and thereforre $\alpha _{n}=0$ for all $n$. In addition, it is easy to see that
all $\beta _{n}$ are real.

To conclude, we need to modify formula (\ref{equ_beta}) for $\beta_n$.
By using property $\langle z^{m + 1}, z^n \rangle = - \langle z^m, z^{n+1}
\rangle$ and the fact that the coefficients of orthogonal polynomials are
real, we see that $\langle P_{n-1}, z P_n \rangle = - \langle z P_{n-1}, P_n
\rangle$.

Since $z P_{n-1} = P_n - \beta_{n - 1} P_{n-2}$, we compute $\langle z
P_{n-1}, P_n \rangle = \langle P_n, P_n \rangle$, and therefore, $\beta_n=
\langle P_n, P_n \rangle / \langle P_{n - 1}, P_{n - 1} \rangle$. This
concludes the proof of the proposition. \hfill $\square $

Table \ref{TablePolynomials} is the table of the first orthogonal monic
polynomials together with recursion coefficients $\beta _{n}$ and the
squared norms of polynomials $h_{n}:=\left\langle P_{n},P_{n}\right\rangle .$

\begin{table}[tbp]
\begin{tabular}{l|l|l|c}
$n$ & $P_n$ & $h_n$ & $\beta_n$ \\ \hline
0 & 1 & 1 & * \\ 
1 & $z$ & 3 & 3 \\ 
2 & $z^2 + 3$ & 6 & 2 \\ 
3 & $z^3 + 5z$ & 30 & 5 \\ 
4 & $z^4 + 10z^2 + 15$ & 120 & 4 \\ 
5 & $z^5 + 14z^3 + 35z$ & 840 & 7 \\ 
6 & $z^6 + 21z^4 + 105z^2 + 105$ & 5,040 & 6 \\ 
7 & $z^7 + 27z^5 + 189z^3 + 315z$ & 45,360 & 9 \\ 
8 & $z^8 + 36z^6 + 378z^4 + 1260z^2 + 945$ & 362,880 & 8 \\ 
9 & $z^9 + 44z^7 + 594z^5 + 2772z^3 + 3465z$ & 3,991,680 & 11%
\end{tabular}%
\caption{Monic orthogonal polynomials, their squared norms and $\protect\beta%
_n$.}
\label{TablePolynomials}
\end{table}

In the next step, we are going to derive more explicit formulas for the
orthogonal polynomials $P_n$, their norms $h_n$, and coefficients $\beta_n$.

\pagebreak

\subsection{Determinantal formulas}

Let $s_{ij}:=\left\langle z^{i},z^{j}\right\rangle .$ (These are elements of
the infinite matrix in Table \ref{table_scalar_products}.) And let $D_{n}$
denote the principal submatrices of the matrix of scalar products: 
\begin{equation*}
D_{n}=\left( 
\begin{array}{cccc}
s_{00} & s_{01} & \cdots & s_{0n} \\ 
s_{10} & s_{11} & \cdots & s_{1n} \\ 
\vdots & \vdots & \ddots & \vdots \\ 
s_{n0} & s_{n1} & \cdots & s_{nn}%
\end{array}%
\right) .
\end{equation*}%
Finally let $\left\vert D_{n}\right\vert $ denotes $\det (D_{n})$.

\begin{theo}
The monic orthogonal polynomials are given by the formula%
\begin{equation*}
P_{n}\left( z\right) =\frac{1}{\left\vert D_{n-1}\right\vert }\det \left( 
\begin{array}{ccccc}
s_{00} & s_{01} & \cdots & s_{0,n-1} & s_{0n} \\ 
s_{10} & s_{11} & \cdots & s_{1,n-1} & s_{1n} \\ 
\vdots & \vdots & \ddots & \vdots & \vdots \\ 
s_{n-1,0} & s_{n-1,1} & \cdots & s_{n-1,n-1} & s_{n-1,n} \\ 
1 & z & \cdots & z^{n-1} & z^{n}%
\end{array}%
\right) .
\end{equation*}%
Their squared norms are $h_{n}:=\left\langle P_{n}\left( z\right)
,P_{n}\left( z\right) \right\rangle =\left\vert D_{n}\right\vert /\left\vert
D_{n-1}\right\vert .$
\end{theo}

\textbf{Proof:} The polynomials are clearly monic. In order to prove
orthogonality, we write 
\begin{equation*}
\left\langle z^{m},P_{n}\left( z\right) \right\rangle =\frac{1}{\left\vert
D_{n-1}\right\vert }\det \left( 
\begin{array}{ccccc}
s_{00} & s_{01} & \cdots & s_{0,n-1} & s_{0n} \\ 
s_{10} & s_{11} & \cdots & s_{1,n-1} & s_{1n} \\ 
\vdots & \vdots & \ddots & \vdots & \vdots \\ 
s_{n-1,0} & s_{n-1,1} & \cdots & s_{n-1,n-1} & s_{n-1,n} \\ 
\left\langle z^{m},1\right\rangle & \left\langle z^{m},z\right\rangle & 
\cdots & \left\langle z^{m},z^{n-1}\right\rangle & \left\langle
z^{m},z^{n}\right\rangle%
\end{array}%
\right) .
\end{equation*}%
This equals $0$ for $m\leq n-1$ because there are two coinciding rows.

For $m=n,$ we have $\left\langle z^{n},P_{n}\left( z\right) \right\rangle
=\left\vert D_{n}\right\vert /\left\vert D_{n-1}\right\vert .$ Since the
polynomials are monic, $\left\langle P_{n}\left( z\right) ,P_{n}\left(
z\right) \right\rangle $ = $\left\langle z^{n},P_{n}\left( z\right)
\right\rangle =\left\vert D_{n}\right\vert /\left\vert D_{n-1}\right\vert .$
\hfill $\square $

\begin{figure}[tbp]
\includegraphics[width=\textwidth]{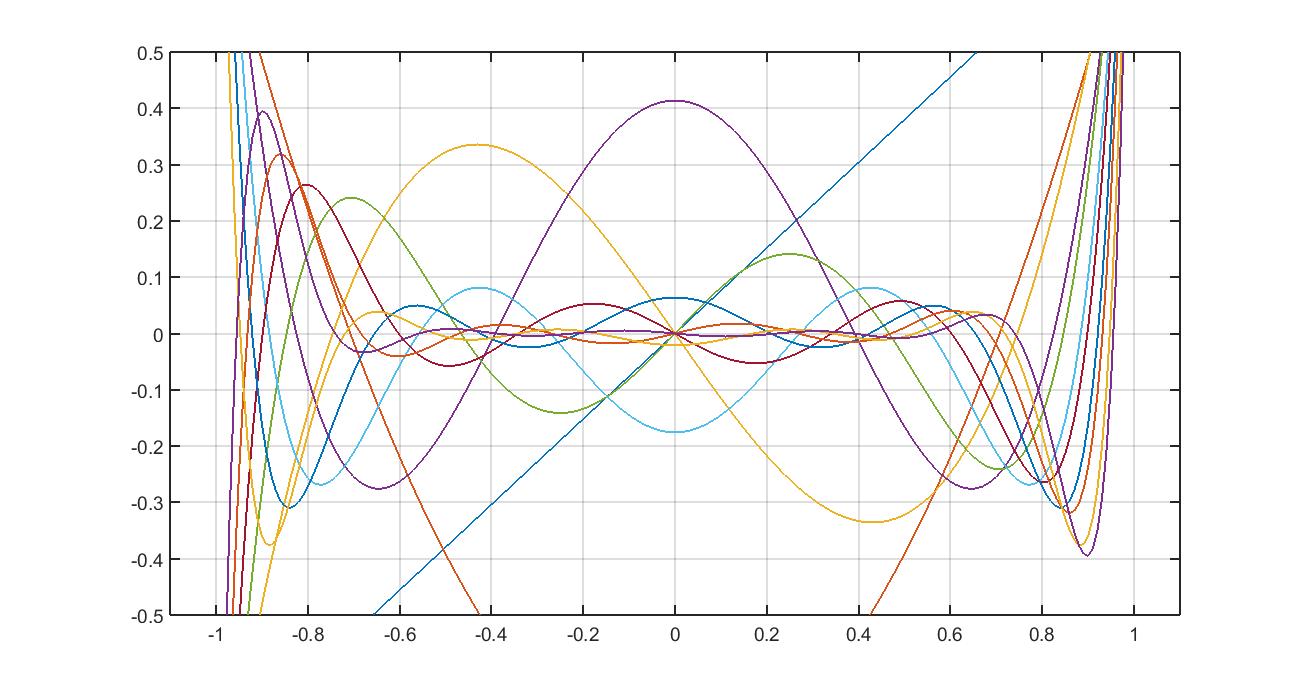}
\caption{Plot of $h_n^{-3/4} P_n(i\protect\sqrt{3n}x)$ 
for $n = 1, \ldots, 9$. The scaling exponent $-3/4$ was chosen ad hoc to fit the plots on the Figure.}
\label{figure_polynomials}
\end{figure}

\pagebreak

\subsection{Norm of polynomials} $\,$ \\

The patterns for $h_n$ and $\beta_n$ are clear from Table \ref%
{TablePolynomials}.

\begin{theo}
\label{h_and_beta} 
\begin{equation*}
h_{n}=\left\{ 
\begin{array}{cc}
n!(n+2), & \text{ if }n\text{ is odd,} \\ 
(n+1)!, & \text{if }n\text{ is even.}%
\end{array}%
\right.
\end{equation*}%
\begin{equation*}
\beta _{n}=\left\{ 
\begin{array}{cc}
n+2, & \text{ if }n\text{ is odd,} \\ 
n, & \text{if }n\text{ is even.}%
\end{array}%
\right.
\end{equation*}
\end{theo}

The proof proceeds by exhibiting an explicit formula for the determinant 
$\vert D_n \vert$. 
Since the proof is rather lengthy, we omit it due to the space constraints.

As a corollary, we find that $\vert D_n \vert > 0$ for all $n$ and, therefore,  all eigenvalues of matrix $D_n$ are positive. After some work, this implies that the scalar product $\scP{u}{v}$ is positive definite on $\MM$.

\pagebreak

%

\subsection{Relation to Hermite polynomials}

%

\theoHermite

\begin{proof}
The first monic Hermite polynomials are $H_{0}=1,$ $H_{1}=x,$ $%
H_{2}=x^{2}-1,$ and $H_{3}=x^{3}-3x,$ and direct verification shows that the statement holds for $n=0$ and $n=1.$ 

We verify by induction that the polynomials on the right hand side (RHS) of 
(\ref{formula_explicit}) satisfy the same 3-term recurrence as $%
P_{n} $. Recall the 3-term recurrence
for Hermite polynomials: $H_{n+1}=xH_{n}-nH_{n-1}.$

For an odd $n$ we need to check that $xP_{n} + (n+2)
P_{n-1}=P_{n+1}.$ Since 
\begin{equation*}
P_n  = \frac{u^n}{s^2}\left( s H_{n+1}+H_n\right) =\frac{u^n}{s^2}\left(
H_{n+2}+(n+2) H_n\right) ,
\end{equation*}%
by hypothesis, we compute left-hand side of the recursion as
\begin{equation*}
\frac{u s u^n}{s^2}\left( H_{n+2}+(n+2)H_n\right)  + \frac{(n+2)u^{n - 1}}{s}H_{n}=
- \frac{u^{n - 1}}{s} H_{n+2} = \frac{u^{n + 1}}{s}H_{n+2},
\end{equation*}%
which is the right hand side. Here we used the fact that $u^2 =  - 1$.  

For an even $n$ we should have $xP_{n} + n P_{n-1} = P_{n+1}$. For the left hand
side, we compute:%
\begin{eqnarray*}
&& u^{n + 1}H_{n+1} + \frac{n u^{n-1}}{s^2} \left( s H_{n} + H_{n-1}\right) \\
&=& u^{n + 1}H_{n+1} + \frac{n u^{n-1}}{s}  H_{n} +  \frac{u^{n-1}}{s^2}\left( s H_n - H_{n+1}\right) \\ 
& = & u^{n + 1} \left( (1 + s^{-2}) H_{n+1} - \frac{n+1}{s} H_n \right),
\end{eqnarray*}%
and for the right hand side: 
\begin{eqnarray*}
& & \frac{u^{n+1}}{s^2}\left( s H_{n+2} + H_{n+1}\right) \\
& = & \frac{u^{n+1}}{s^2}\left( s^2 H_{n+1} - (n+1)s H_n +  H_{n+1}\right) \\
& = & u^{n + 1} \left( (1 + s^{-2}) H_{n+1} - \frac{n+1}{s} H_n \right).
\end{eqnarray*}
Together,
these equalities verify the desired recursion identity.
\end{proof}

%

\subsection{$Q$ polynomials}

Let $s$ be real. Then we can define the polynomials $Q(s)$ by the formula:

\begin{restatable}{equation}{equationPQ}
P_{n}(u s) = u^n Q(s).
\end{restatable}

The first ten $Q_{n}$ are shown in Table \ref{TableQ}.

\begin{table}[tbp]
\begin{tabular}{l|l}
$n$ & $Q_n(x)$ \\ \hline
0 & 1 \\ 
1 & $x$ \\ 
2 & $x^2 - 3$ \\ 
3 & $x^3 - 5x$ \\ 
4 & $x^4 - 10x^2 + 15$ \\ 
5 & $x^5 - 14x^3 + 35x$ \\ 
6 & $x^6 - 21x^4 + 105x^2 - 105$ \\ 
7 & $x^7 - 27x^5 + 189x^3 - 315x$ \\ 
8 & $x^8 - 36x^6 + 378x^4 - 1260x^2 + 945$ \\ 
9 & $x^9 - 44x^7 + 594x^5 - 2772x^3 + 3465x$%
\end{tabular}%
\caption{$Q_n(x)$ polynomials}
\label{TableQ}
\end{table}

\begin{theo}
(i) Polynomials $Q_{n}(x)$ satisfy the following recursion: 
\begin{equation}
Q_{n+1}\left( x\right) =xQ_{n}\left( x\right) -\beta _{n}Q_{n-1}\left(
x\right) ,  \label{ThreeTermQ}
\end{equation}%
\newline
(ii) Polynomials $Q_{n}(x)$ are orthogonal with respect to a non-negative
measure $\nu $ on $\mathbb{R}$.\newline
(iii) The coefficients of every polynomial $Q_{n}(x)$ are real.\newline
(iv) All the zeros of a polynomial $Q_{n}\left( x\right) $ are simple and
real.\newline
(v) \ Any two zeros of a polynomial $Q_{n}\left( x\right) $ are separated by
a zero of polynomial $Q_{n-1}\left( x\right) $ and vice versa.
\end{theo}

\begin{proof} Formula (\ref{ThreeTermQ}) follows from the properties of $%
P_{n}(x).$ Claim (ii) follows by Favard's theorem, because $\beta _{n}$ are
positive. Claim (iii) is implied by (i) becase $\beta _{n}$ are real. Claims
(iv) and (v) are implied by (ii), see Theorem 1.2.2 in Akhieser \cite%
{akhieser61}. 
\end{proof}

\begin{theo}
The polynomials $Q_{n}\left( x\right) $ are monic orthogonal polynomials
with respect to measure $\nu $ with density 
\begin{equation*}
f(t)=\frac{t^{2}}{\sqrt{2\pi }}e^{-t^{2}/2},
\end{equation*}%
defined on all real line.
\end{theo}

\begin{proof} The moments of this measure are $m_{2k+1}=0$ and $%
m_{2k}=(2k+1)!!.$ By using these moments, we can calculate the coefficients
in the 3-term recurrence relation for orthogonal polynomials related to this
measure. This is done as in the proof of Proposition \ref{h_and_beta}. It
turns out that these coefficients are the same as for the polynomials $Q_{n}$.
 Since the initial conditions are also satisfied, $Q_{n}$ are the monic
orthogonal polynomials for the measure $\nu$.
\end{proof}

%
\section{Kernel}
\label{section_kernel}

\subsection{Definition}

Define 
\begin{equation*}
K_{n}\left( x,y\right) =\sum_{k=0}^{n}\frac{P_{k}\left( x\right) \overline{%
P_{k}\left( y\right) }}{h_{k}}.
\end{equation*}
As we mentioned before, this is a valid kernel for the quaternion field with respect to the
background measure on $\R^3$, $d\mu \left( z\right) =\left( 2\pi \right)
^{-3/2}e^{-\left\vert z\right\vert ^{2}/2}dm(z)$, where $dm(z)$ denotes the
Lebesgue measure on $\mathbb{R}^3.$

%

\subsection{Formulae for the kernel}

%

\subsubsection{Christoffel-Darboux relation}

\begin{theo}
The following Christoffel-Darboux formula holds for all $n\geq 0,$ and all
imaginary quaternions $x$ and $y$:%
\begin{equation}
xK_{n}\left( x,y\right) +K_{n}\left( x,y\right) \overline{y}=\frac{%
P_{n+1}\left( x\right) \overline{P_{n}\left( y\right) }+P_{n}\left( x\right) 
\overline{P_{n+1}\left( y\right) }}{h_{n}}.  \label{ChristoffelDarboux}
\end{equation}
\end{theo}

\textbf{Proof:} By using the three-term recurrence relation (\ref%
{ThreeTermRelation}) and setting $P_{-1}\left( x\right) =0$, we can write%
\begin{equation*}
xK_{n}\left( x,y\right) =\sum_{k=0}^{n}\frac{(P_{k+1}\left( x\right) -\beta
_{k}P_{k-1}\left( x\right) )\overline{P_{k}\left( y\right) }}{h_{k}}.
\end{equation*}%
Similarly, 
\begin{equation*}
K_{n}\left( x,y\right) \overline{y}=\sum_{k=0}^{n}\frac{P_{k}\left( x\right) 
\overline{yP_{k}\left( y\right) }}{h_{k}}=\sum_{k=0}^{n}\frac{P_{k}\left(
x\right) \left( \overline{P_{k+1}\left( y\right) }-\beta _{k}\overline{%
P_{k-1}\left( y\right) }\right) }{h_{k}}.
\end{equation*}

If we add these two expressions together and use the fact that $\beta
_{k}=h_{k}/h_{k-1},$ then we obtain equation (\ref{ChristoffelDarboux}). $%
\square $

%

\subsubsection{Explicit expression for the kernel}


%

 We solve equation (\ref%
{ChristoffelDarboux}) by using a result from Janovsk\'{a} and
Opfer \cite{Janovska_Opfer08}.

An equation $Ax+xB=C$ is called \emph{singular} if the equation $Ax+xB=0$
has a non-zero solution.

%

\begin{lemma}[Janovska-Opfer]
\label{JanovskaOpfer} The equation 
\begin{equation*}
Az+zB=C;\text{ \ \ }A,B,C,z\in \mathbb{H}\text{, }AB\neq 0,
\end{equation*}%
is singular if and only if $A$ and $-B$ are equivalent (that is, $\left\vert
A\right\vert =\left\vert B\right\vert ,$ and $\mathrm{Re}(A)=\mathrm{Re}(-B)$%
). If it is non-singular, then its solution is 
\begin{equation}
z=f_{l}^{-1}\left( C+A^{-1}C\overline{B}\right) =(C+\overline{A}%
CB^{-1})f_{r}^{-1},  \label{equation_solution}
\end{equation}%
where $f_{l}=2\mathrm{Re}\left( B\right) +A+\left\vert B\right\vert
^{2}A^{-1}, $ and $f_{r}=2\mathrm{Re}(A)+B+\left\vert A\right\vert
^{2}B^{-1}.$
\end{lemma}

\textbf{Sketch of the proof: }Assume that $A$ and $B$ are not equivalent.
This implies that $f_{l}$ and $f_{r}$ are not zero. Then, for the first
equality in (\ref{equation_solution}), we have \textbf{\ }%
\begin{eqnarray*}
C+A^{-1}C\overline{B} &=&Ax+xB+A^{-1}\left( Ax+xB\right) \overline{B} \\
&=&x\left( B+\overline{B}\right) +Ax+A^{-1}x\left\vert B\right\vert ^{2} \\
&=&\left( 2\mathrm{Re}\left( B\right) +A+\left\vert B\right\vert
^{2}A^{-1}\right) x.
\end{eqnarray*}%
Hence $x=f_{l}^{-1}\left( C+A^{-1}C\overline{B}\right) .$ The second
equality is proved in a similar way. \hfill $\square$

%

\theoKernelFormula

\begin{proof} We apply Theorem \ref{JanovskaOpfer} with $z=K_{n}\left(
x,y\right) ,$ $A=x,$ $B=\overline{y},$ and 
$C_{n}=h_n^{-1}(P_{n+1}(x) \overline{P_{n}(y)}
+ P_{n}(x) \overline{P_{n+1}( y)}).$

Using the fact that $x$ and $y$ are purely imaginary, and therefore $\mathrm{Re}(x)=\mathrm{Re}%
(y)=0,$ $\overline{x}=-x$, $\overline{y}= - y$, we calculate

\begin{align*}
f_l &= x + |\overline{y}|^2 x^{-1},\\
f_l^{-1} &= \frac{x}{x^2 - y^2}, \\
K_n(x, y) &= \frac{x}{x^2 - y^2}(C + x^{-1} C y).
\end{align*}

Hence, the solution is 
\begin{equation*}
K_{n}\left( x,y\right) =\frac{1}{x^{2}-y^{2}}\left( xC_{n} + C_{n}y \right) .
\end{equation*}

Next, we substitute the expression for $C_n$ and obtain the equation 
\begin{align*}
K_n(x, y) &= \frac{1}{h_{n}(x^2-y^2)}
\bigg(
x 
\Big[
P_{n+1}(x) \overline{P_n(y)} + P_n(x) \overline{P_{n+1}(y)}
\Big] \\
&+ 
\Big[
P_{n+1}(x) \overline{P_n(y)} + P_n(x) \overline{P_{n+1}(y)}
\Big]
y
\bigg).
\end{align*}

Next, let $x = u s$ and $y = v t$, where $s = |x|$ and $t = |y|$. 

Then, after some calculations we find that for odd $n$

\begin{align*}
K_n(x, y) = \frac{1}{h_n(t^2-s^2)}
\Big(
Q_n(s) Q_{n+1}(t) \big( s - u v t\big) 
- Q_{n+1}(s) Q_n(t) \big( t - u v s\big)
\Big).
\end{align*}%

 For even $n$,
\begin{align*}
K_n(x, y) = \frac{1}{h_n(t^2-s^2)}
\Big(
Q_n(s) Q_{n+1}(t) \big( t - u v s\big) 
- Q_{n+1}(s) Q_n(t) \big( s - u v t\big)
\Big).
\end{align*}
The statement of the theorem follows after a term rearrangement.
\end{proof}

\begin{figure}[tbp]
\includegraphics[width=\textwidth]{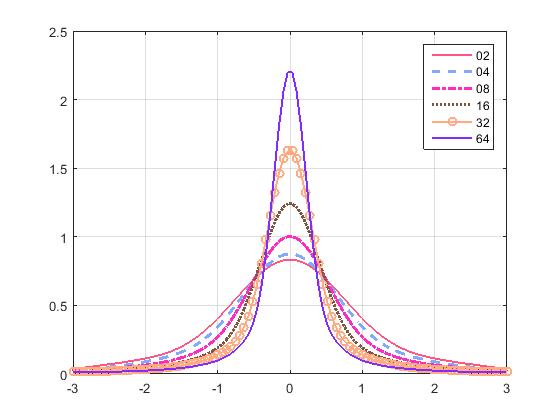}
\caption{Plot of $\frac{1}{n + 1}K_n(i s, i s) e^{-s^2/2} $ for $n = 2, 4,
\ldots, 64$.}
\label{figure_correlation1st}
\end{figure}

From this expression for the kernel, we can derive formulae for the first and second correlation functions. 

\begin{coro} 
Let $x = u s$ and $y = v t$, where $s = |x|$ and $t = |y|$. Then, 
\begin{enumerate}
\item The first correlation function (i.e., the density of the point field with respect to the background measure) is 
\[
p_n^{(1)}(x) = \rho_n(s) := \frac{1}{h_n} \big[ Q_n(s)Q_{n+1}^{\prime }(s) 
- Q_{n+1}(s) Q_n^{\prime }(s) \big].
\]
\item The second correlation function is 
\[
p_n^{(2)}(x, y) = \rho_n(s) \rho_n(t) 
- \Big[ \rho_n(s,t)^2 \frac{1 + \cos \alpha}{2} 
+ \delta_n(s,t)^2 \frac{1 - \cos \alpha}{2} \Big], 
\]
where $\alpha$ is the angle between vectors $x$ and $y$.
\end{enumerate}
\end{coro}

The first correlation function is shown in Figure \ref{figure_correlation1st}.

If $t = s$, then  we obtain the following expression for the kernel:
\begin{align*}
K_n(u s, v s) = \frac{1}{2}\big[\rho_n(s) -\delta_n(s)\big] 
 -\frac{1}{2}\big[\rho_n(s) + \delta_n(s)\big] u v,
\end{align*}
where 
\begin{align}
\label{def_rho_s}
\delta_n(s) &= K_n(u s, - u s) =  \frac{1}{s \, h_n}  Q_n(s)Q_{n+1}(s). 
\end{align}

And in this case, the second correlation function is 
\begin{align*}
p_n^{(2)}(us, vs) & = \det 
\begin{bmatrix}
K_n(us, us) & K_n(us, vs) \\
K_n(vs, us)  & K_n(vs, vs)
\end{bmatrix}
 \\
& = \big[\rho_n(s)^2 - \delta_n(s)^2 \big]
\frac{1 - \cos \alpha}{2}, 
\end{align*}
where $\alpha$ is the angle between $u$ and $v$. This expression coincide with $\rho_n(s)^2$ (which would be valid for the Poisson random field with density $\rho_n(s)$ only if $\alpha = \pi$.

In addition, as a function of $u$ and $v$, the kernel $K_n(u s, v s)$ has rank 2. Hence, all determinants $\det\Big([K_n(u_i s, u_j s)]_{i, j = 1}^{r}\Big)$ are zero for $r \geq 3$. It follows that all correlation functions of order higher than two vanish. 

This means that it is not possible to have more than two particles on any sphere, and by continuity of correlation functions, it is not likely to have three or more points at approximately the same distance from the origin. 

This property is similar to the Pauli exclusion principle in quantum mechanics, which states that no two fermion particles can be in an identical state. In particular, each state in an atom can hold only two electrons that differ by their spins. 

As another remark, note that this determinantal point field, conditioned to be on the sphere, has the same correlation structure as the spherical determinantal field with two points. This field (in general with $N$ points) was discovered by   Caillol \cite{caillol81}  and rediscovered recently by Krishnapur \cite{krishnapur08} in his extensive study of determinantal point fields.

For the collinear and anti-collinear points ($u = \pm v$) we have 
\bal{
p_n^{(2)}(us, ut) & = \rho_n(s) \rho_n(t) - \rho_n(s, t)^2 \\
p_n^{(2)}(us, -ut) & = \rho_n(s) \rho_n(t) - \delta_n(s, t)^2.
}

%
\newpage
\subsection{Rotational Invariance of the Kernel}
 
 Define the following action of $\bH \setminus \{0\}$
on $\Lambda$.
\[
\Ad_q x = q x q^{-1}.
\]
Recall that this action represents three-dimensional rotations. Namely, if 
$x = x_1 \bi + x_2\bj + x_3\bk$, and $q$ is the unit quaternion 
$q = \cos \theta + (q_1 \bi + q_2 \bi + q_3 \bi) \sin \theta$, then 
$\Ad_q x = y_1 \bi + y_2\bj + y_3\bk$, and the vector  $(y_1, y_2,  y_3)$ is the image of the vector $(x_1, x_2,  x_3)$ after a rotation around the axis with
direction vector $(q_1, q_2, q_3)$ by the angle $\theta$.

\begin{coro}
\label{coro_invariance}
 Let $x$ and $y$ be imaginary quaternions. 
 Then the following rotational invariance holds:
\begin{align*}
K_n(\Ad_q x,\Ad_q y) &= K_n(x, y). 
\end{align*}
\end{coro}

\begin{proof}
This follows directly from Theorem \ref{theoKernelFormula}.  
\end{proof}

%

\subsection{Kernel Asymptotics in the Bulk}

What happens when the number of particles $n$ grows?  

It turns out that almost all particles are contained in the ball of the radius 
$2\sqrt{n}$, and it is convenient to introduce new variables  $\phi$, $\psi$ in the following way, 
\bal{
\phi &= \arccos \Big(\frac{s}{2\sqrt{n+3/2}}\Big), \\
\psi &= \arccos \Big(\frac{t}{2\sqrt{n+3/2}}\Big).
}

In addition, let us introduce the following notation:

\begin{restatable}{equation}{equAC}
\begin{split}
\ac{x} &:=  \sin(2 x) - 2 x, \\
\daleth_n(\phi,\psi) & := \frac{n + 3/2}{2} (\ac{\phi} - \ac{\psi}), \\
\shin_n(\phi,\psi) & = \frac{n + 3/2}{2} (\ac{\phi} + \ac{\psi}).  
\end{split}
\end{restatable}

Due to the formula for the kernel in Theorem \ref{theoKernelFormula}, it is enough to consider the asymptotic behavior of functions $\rho_n(s, t)$ and $\delta_n(s, t)$.

%

\begin{theo}
\label{theo_kernel_asymp}
Let $s, t \in [\eps \sqrt{n}, (2 - \eps) \sqrt{n}]$ and $s\neq t.$ Then,

\begin{align*}
\rho_n(s, t) e^{- \frac{s^2 + t^2}{4}}  
&= \sqrt{\frac{2}{\pi}}\frac{1}{8}\Big(n+\frac{3}{2}\Big)^{-3/2}
\frac{1}{(\cos \phi - \cos \psi) \cos \phi \cos \psi \sqrt{\sin \phi \sin \psi} }
 \\
 \times \frac{1}{2} & \Big\{  \cos (\daleth_n(\varphi,\psi) - \phi ) 
  - \cos(\daleth_n(\phi, \psi) + \psi)   \\
   &- \sin(\shin_n(\phi,\psi) - \psi)    + \sin(\shin_n(\phi,\psi) - \phi) 
+ O_{\eps}\Big( \frac{1}{\sqrt{n}} \Big) \Big\} .
\end{align*}

and
\bal{
\label{formula_kernel_delta} 
\delta_n(s, t) e^{-\frac{s^{2}+t^{2}}{4}} = O_{\eps}\Big( \frac{1}{n} \Big).
}
\end{theo}

Note that while the estimate for $\delta_n(s, t)$ is uniform in $s$ and $t$, so that this part of the kernel becomes negligible as $n$ grows to infinity, the estimate for $\rho_n(s,t)$ is non-uniform. Its main term can become large if $s$ and $t$ are sufficiently close to each other. 

\begin{proof}
We can express $Q$-polynomials in terms of Hermite polynomials by using Theorem \ref{theoHermite} and formula

\equationPQ*
  Then, by using the Plancherel-Rotach asymptotic formulas for
the Hermite polynomials (see Appendix \ref{asymptoticFormulas} and Szego \cite{szego67}, Theorem 8.22.9), we can derive the asymptotic
expressions for the $Q$-polynomials and for the functions $\rho_n(s, t)$ and 
$\delta_n(s,t)$.

For $\delta_n(s,t)$, this calculation leads to the conclusion that $\delta_n(s,t) = O(n^{-1})$. For 
$\rho_n(s,t)$, the situation is more complicated. First, we get 

\begin{equation}
\label{formula_kernel} 
\begin{split}
\rho_n(s, t) e^{-\frac{s^{2}+t^{2}}{4}} =&\sqrt{\frac{2}{\pi }}%
\frac{1}{(s - t) s t}
\bigg[ (\sin \varphi _{n+2}\sin \psi
_{n+1}) ^{-1/2}
  \\
&\times \sin \left( \Big(\frac{n}{2}+\frac{5}{4}\Big) \ac{\phi _{n+2}} +\frac{3\pi }{4}\right)
 \sin \left( \Big( \frac{n}{2}+\frac{3}{4}\Big) \ac{\psi_{n+1}} +\frac{3\pi }{4}\right)    
\\
&-\left( \sin \varphi _{n+1}\sin \psi _{n+2}\right) ^{-1/2}
 \\
&\times \sin\left( \Big(\frac{n}{2} + \frac{5}{4}\Big) \ac{\psi_{n+2}} +\frac{3\pi }{4}\right) \sin \left( \Big( \frac{n}{2}+\frac{3}{4}\Big) \ac{\phi _{n+1}} +\frac{3\pi }{4}\right) \bigg] 
\\
&+O\Big( \frac{1}{\sqrt{n}}\Big),
\end{split}
\end{equation}
where 
\bal{
\phi_{n+1} &= \arccos\Big(\frac{x}{2\sqrt{n + 3/2}}\Big),\\
\phi_{n+2} &= \arccos\Big(\frac{x}{2\sqrt{n + 5/2}}\Big),
}
and similar expressions hold for $\psi_{n+1}$ and $\psi_{n+2}$.

 In order to get a bit simpler expression, we express $\phi_{n+2}$ and $\psi_{n+2}$ in terms of $\phi_{n+1}$ and $\psi_{n+1}$ and expand this expression in powers of $n^{-1}.$ (For parsimony,  we
write $\phi_1$, $\phi_2$, $\psi_1$ and $\psi_2$ for $\phi_{n+1}$,  $\phi_{n+2}$, 
$\psi_{n+1}$  and $\psi_{n+2}$, respectively.) 

\begin{equation*}
\varphi_2 = \varphi_1 + \frac{\cos \varphi_1}{\sin \varphi_1}
\left( \frac{1}{2 n} - \frac{9}{8 n^2} \right) 
- \frac{1}{8} \frac{\cos^3 \varphi_1}{\sin^3 \varphi_1} \frac{1}{n^2} 
+ O\left(\frac{1}{n^3}\right)
\end{equation*}
A similar formula holds for $\psi_2$.

Then,
\begin{equation}
\label{equ_phi2_phi1}
\begin{split}
\ac{\varphi_2} & = \ac{\varphi_1} 
+ 2 [\cos (2 \varphi_1) - 1] \frac{\cos \varphi_1}{\sin \varphi_1} \frac{1}{2 n}  
+ O\left(\frac{1}{n^2}\right)
\\
 & = \ac{\varphi_1} 
 -   \frac{\sin(2 \varphi_1)}{n}  
+ O\left(\frac{1}{n^2}\right),
\end{split}
\end{equation}
and a similar formula holds for $\ac{\psi_2}$.

By using the elementary trigonometric identities \\
 $\sin \alpha \sin \beta = \frac{1}{2}\left(\cos(\alpha - \beta) - \cos(\alpha + \beta) \right]$ and
 $\cos(x+3\pi/2) = \sin x$, 
we re-write several terms in (\ref{formula_kernel}):

\begin{eqnarray*}
& &\sin \left[ \left(\frac{n}{2} + \frac{5}{4} \right) \ac{\varphi_2} + \frac{3 \pi}{4} \right]
\sin \left[ \left(\frac{n}{2} + \frac{3}{4} \right) \ac{\psi_1} + \frac{3 \pi}{4} \right] \\
& = & \frac{1}{2} \left\{ 
\cos \left[ \Big(\frac{n}{2} + \frac{3}{4} \Big) (\ac{\varphi_2} - \ac{\psi_1}) + \frac{1}{2}\ac{\varphi_2} \right] \right.
 \\
& - & \left. \sin \left[ \Big(\frac{n}{2} + \frac{3}{4} \Big) (\ac{\varphi_2} + \ac{\psi_1}) + \frac{1}{2}\ac{\varphi_2}  \right] \right\},
\end{eqnarray*}

and 

\begin{eqnarray*}
& &\sin \left[ \Big(\frac{n}{2} + \frac{3}{4} \Big) \ac{\phi_1} + \frac{3 \pi}{4} \right]
\sin \left[ \Big(\frac{n}{2} + \frac{5}{4} \Big) \ac{\psi_2} + \frac{3 \pi}{4} \right] \\
& = & \frac{1}{2} \left\{ \cos \left[ \Big(\frac{n}{2} + \frac{3}{4} \Big) (\ac{\phi_1} - \ac{\psi_2}) - \frac{1}{2}\ac{\psi_2} \right] \right.
 \\
& - & \left. \sin \left[ \Big(\frac{n}{2} + \frac{3}{4} \Big) (\ac{\phi_1} + \ac{\psi_2}) + \frac{1}{2}\ac{\psi_2} \right] 
\right\}.
\end{eqnarray*}

Next, note that both $(\sin \phi_2 \sin \psi_1)^{-1/2}$ and $(\sin \phi_1 \sin \psi_2)^{-1/2}$ in formula (\ref{formula_kernel}) equal $(\sin \phi_1 \sin \psi_1)^{-1/2} + O(n^{-1})$. Hence, we can combine terms and use (\ref{equ_phi2_phi1}) to obtain
\begin{align*}
 &\cos \Big( \frac{2 n + 3}{4}  (\ac{\phi_2} - \ac{\psi_1}) + \frac{1}{2}\ac{\phi_2} \Big) 
  -  \cos \Big( \frac{2 n + 3}{4} (\ac{\phi_1} - \ac{\psi_2}) - \frac{1}{2}\ac{\psi_2} \Big) \\
  & = \cos \Big (\frac{2n+3}{4}( \ac{\phi_1} - \ac{\psi_1}) - \phi_1 \Big) 
  -   \cos \Big(\frac{2n+3}{4}( \ac{\phi_1} - \ac{\psi_1}) + \psi_1 \Big) + O\Big(\frac{1}{n}\Big), 
\end{align*}

and 

\begin{align*}
 &\sin \Big( \frac{2 n + 3}{4}  (\ac{\phi_1} + \ac{\psi_2}) + \frac{1}{2}\ac{\psi_2} \Big) 
  -  \sin \Big(\frac{2 n + 3}{4} (\ac{\phi_2} + \ac{\psi_1}) + \frac{1}{2}\ac{\phi_2}) \\
  & = \sin \Big(\frac{2n+3}{4}(\ac{\phi_1} + \ac{\psi_1}) - \psi_1 \Big) 
  -   \sin \Big(\frac{2n+3}{4}(\ac{\phi_1} + \ac{\psi_1}) - \phi_1 \Big) + O\Big(\frac{1}{n}\Big). 
\end{align*}

With these modifications, formula (\ref{formula_kernel}) implies the statement of the theorem. 
\end{proof}

\begin{figure}[tbp]
\includegraphics[width=\textwidth]{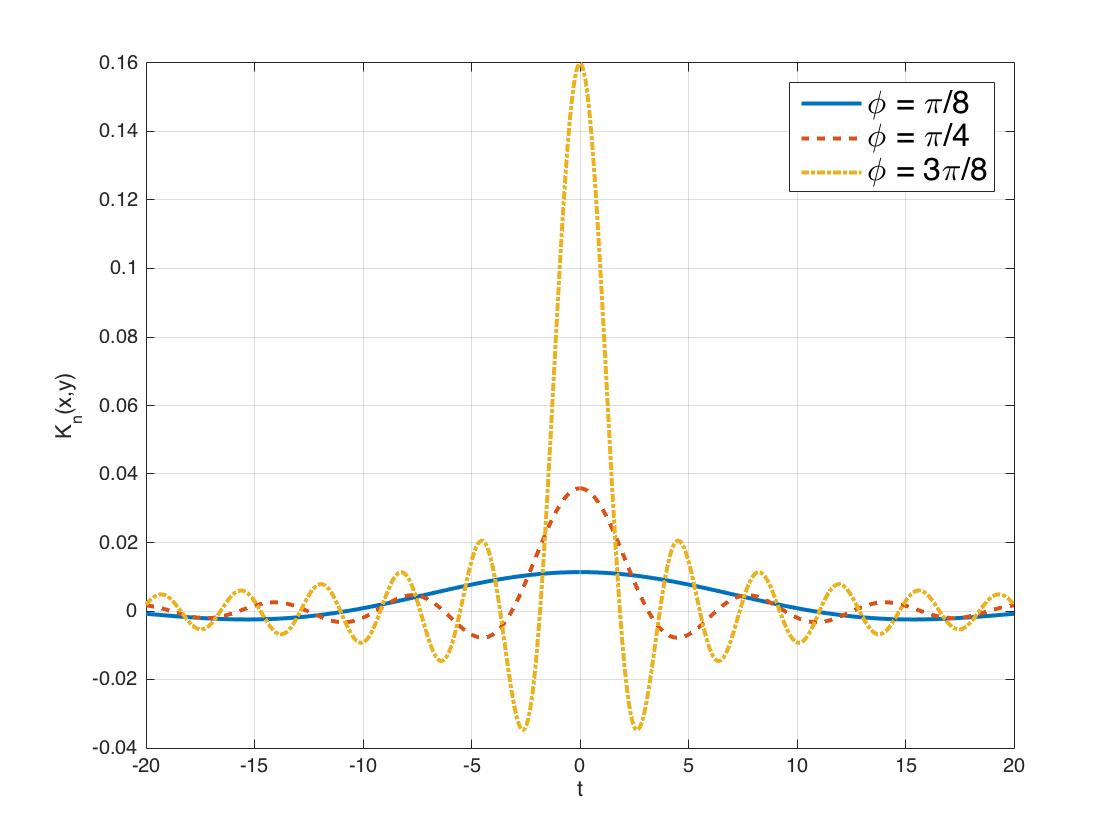}
\caption{A plot of
$\lim_{n \to \infty} \hat{K}_n(x, x + \frac{\tau}{n})$ for various values of parameters
 $\phi = \arccos x$ and $\tau$.}
\label{figure_correlations}
\end{figure}

%

Let 

\defiXPhi*

And let

\defiSnTn*

\begin{theo}
\label{theo_rescaled_rho}
For $x_0 \in [\eps, 1 - \eps]$, and $s_n$, $t_n$ as defined above, we have
\bal{
2\sqrt{n + \frac{3}{2}}\rho_n(s_n, t_n) \sqrt{f(s_n)f(t_n)}
 = \frac{\sin (\tau - \sigma)}{\tau - \sigma}\rho(x_0) + O_{\eps}(n^{-1/2}) .
}
\end{theo}

Here $f(s)$ is the density of the background measure as in (\ref{background_density}).

\begin{proof} 
We use the formula from Theorem \ref{theo_kernel_asymp}. 
 Let us for shortness write $k$ for $\frac{1}{2\sqrt{1 - x_0^2}}$ and note that 
 \bal{
  \sin\Big(\shin_n - \phi - \frac{k\tau}{n}\Big)  
  -  \sin\Big(\shin_n - \phi - \frac{k\sigma}{n}\Big) 
   &= \sin (\shin_n - \phi)\Big(\cos \frac{k\tau}{n} - \cos \frac{k\sigma}{n}\Big)
  \\
  & - \cos (\shin_n - \phi) \Big(\sin \frac{k\tau}{n} - \sin \frac{k\sigma}{n}\Big)
\\
  &= O_{\eps}\big(n^{-1}\big).
  }
Next, we have 
  
 \bal{
 \daleth_n\Big(\phi + \frac{k\sigma}{n}, \phi + \frac{k\tau}{n}\Big) 
 &= \frac{n + \frac{3}{2}}{2}
 \Bigg(\ac{\phi + \frac{k\sigma}{n}} - \ac{\phi + \frac{k\tau}{n}}\Bigg) \\
 &=  \frac{n + \frac{3}{2}}{2}\Big(\sin 2\big(\phi + \frac{k\sigma}{n}\big)- \sin 2\big(\phi + \frac{k\tau}{n}\big) \\
 & + 2\frac{k(\tau - \sigma)}{n}\Big) 
 \\
 &= (1 - \cos 2\phi) k (\tau - \sigma) + O_\eps\big( n^{-1}\big) \\
 &=  \tau - \sigma + O_\eps\big( n^{-1}\big). 
 } 
 
 Hence, 
 \bal{
 &\cos \big(\dalet_n - \phi - \frac{k\sigma}{n} \big) 
 - \cos\big(\daleth_n + \phi + \frac{k\tau}{n}\big) \\
 & =  \cos(\tau - \sigma - \phi) 
- \cos(\tau - \sigma + \phi) + O_\eps(n^{-1}) \\
& = 2 \sin (\tau - \sigma) \sin \phi  + O_\eps(n^{-1}).
 }
 
 Therefore, after some calculation we get from Theorem \ref{theo_kernel_asymp}, 
 
 \bal{
  2 \sqrt{n + \frac{3}{2}}\rho_n(s_n, t_n) \Big(\frac{1}{2\pi }\Big)^{3/2}e^{- \frac{x_n^2 + y_n^2}{4}}  
\\
 =\frac{1}{(2\pi)^2} \frac{\sin (\tau - \sigma)}{\tau - \sigma} \frac{\sin \phi}{\cos^2 \phi} 
   + O_\eps(n^{-1/2}) . 
}

 \end{proof}

\theoDensity

\begin{proof}
This theorem immediately follows from Theorem \ref{theo_rescaled_rho} if we take $\sigma$ and $\tau$ to zero.
\end{proof}

\coroRadialDensity

\begin{proof}
Note that 
\bal{
\pp_n(r) := \int_{r S^2} \rho_n(x) \, d\mu(x) = \int_{S^2} r^2 \rho_n(x) \, d\mu(x).
}
Hence,
\bal{
\frac{1}{2 \sqrt{n}}\pp_n(2 \sqrt{n} s) 
&= \frac{1}{2 \sqrt{n}} \int_{S^2} 4n s^2 \rho_n(2 \sqrt{n} u s) \, d\mu(u) \\
&= 4\pi \times 2 \sqrt{n} \rho_n(2 \sqrt{n} u s) \\
& \to \frac{1}{\pi} \sqrt{1 - s^2}.
} 

\end{proof}

\theoBulkAsymptotic
\begin{proof}
From Theorem \ref{theo_kernel_asymp}, we see that 
$\sqrt{n}\delta_n(s_n, t_n) = O\big(n^{-1/2}\big)$. Then, the claim of the theorem follows by combination of results in Theorems \ref{theoKernelFormula} and \ref{theo_rescaled_rho}. 
\end{proof}

%

\subsection{Kernel Asymptotics at the Center}

Here we study the kernel near the center without rescaling by $\sqrt{n}$.

\theoKernelCenter

\begin{proof}

We use Theorems \ref{theoHermite} and \ref{theoKernelFormula} to get an expression for $\rho_n(s, t)$ in terms of Hermite polynomials. 

Let $n$ be odd, then
\bal{
h_n \rho_n(s, t) &=  \frac{1}{s t} H_{n + 1}(s) H_{n + 1}(t)
 - \frac{1}{s^2 t^2} (n + 1) H_{n}(s) H_{n}(t)  \\
&+ \frac{1}{(t - s) s^2 t} \big[t + (n + 1) s\big] H_{n}(s) H_{n + 1}(t) \\
&- \frac{1}{(t - s) s t^2} \big[t (n + 1) + s\big] H_{n+1}(s) H_{n}(t) \\
}
and a similar formula holds for the even $n$.

Next we use an approximation for Hermite polynomials from Theorem \ref{theoHermiteApproximation} and find that 
\bal{
\rho_n(s, t) e^{-\frac{s^2 + t^2}{4}} 
= \sqrt{\frac{2}{\pi}} \frac{1}{s t}\Big(\frac{\sin\big(\sqrt{n}(t - s)\big)}{(t - s) } + O(n^{-1/2})\Big).
}
The same formula holds for an even $n$.

For $\delta_n(s, t)$ we have the formula: 

\bal{
h_n \delta_n(s, t) &=  \frac{1}{s t} H_{n + 1}(s) H_{n + 1}(t)
 - \frac{1}{s^2 t^2} (n + 1) H_{n}(s) H_{n}(t)  \\
&+ \frac{1}{(t + s) s^2 t} \big[t + (n + 1) s\big] H_{n}(s) H_{n + 1}(t) \\
&- \frac{1}{(t + s) s t^2} \big[t (n + 1) + s\big] H_{n+1}(s) H_{n}(t) 
}
for odd $n$ and a similar formula holds for even $n$. 

By applying approximation from Theorem \ref{theoHermiteApproximation}, we find that 
\bal{
\delta_n(s, t) e^{-\frac{s^2 + t^2}{4}} 
= \sqrt{\frac{2}{\pi}} \frac{1}{s t}\Big(\frac{\sin\big(\sqrt{n}(t + s)\big)}{(t + s) } + O(n^{-1/2})\Big).
}
\end{proof}

 For the density we have 
 
\bal{
\rho_n(s) e^{-\frac{s^2}{2}} 
= \sqrt{\frac{2}{\pi}} \frac{1}{s^2} \Big(\sqrt{n} + O(n^{-1/2})\Big).
}

%

\section{Open Problems}
\label{sectionProblems}
Several problems seems to be interesting to investigate.

\begin{enumerate}
\item What is the size of holes in this field? That is, if we select a point 
$x \in \R^3$, then how far is the closest field point?

\item It would be interesting to see if this field can be realized as a field of eigenvalues of some quaternion matrices. 

\item Is there a random dynamical system, for which the field distribution is an equilibrium distribution? Here we have in mind a system which would resemble the Dyson Brownian motion.

\item Is it possible to adapt the constructions in this paper to define a random field on $R^4$?
\end{enumerate}

%
\newpage
\appendix

\section{Useful asymptotic formulas}
\label{asymptoticFormulas}

The following is a modification of the standard Plancherel - Rotach for the version of Hermite polynomials that we use in this paper. 

\begin{theo}[Plancherel - Rotach]
Let $\epsilon$ and $\omega$ be fixed positive numbers. We have \\
(a) for $x = 2\sqrt{n + \frac{1}{2}} \cos \phi$, $\epsilon \leq \phi \leq \pi - \epsilon$, 
\begin{align*}
e^{-x^2/4} H_n(x) &= \Big(\frac{2}{\pi n}\Big)^{\frac{1}{4}}\sqrt{n!} \frac{1}{\sqrt{\sin \phi}} \\
&\cdot\bigg\{ \sin \Big[ \big(\frac{2n + 1}{4}\big)\big( \sin\, 2\phi - 2 \phi \big) + \frac{3 \pi}{4}\Big] + O(n^{-1})\bigg\}.
\end{align*}
(b) for $x =  2\sqrt{n + \frac{1}{2}} - \big(\frac{1}{9 n}\big)^{1/6} t$, $t$ complex and bounded,
\begin{align*}
e^{-x^2/4} H_n(x) = \Big(\frac{8\cdot 81}{\pi^9 n}\Big)^{\frac{1}{12}}\sqrt{n!} \Big\{ A(t) + O(n^{-2/3}) \Big\},
\end{align*}
where $A(t)$ is the Airy function.
\end{theo}

\begin{theo}
\label{theoHermiteApproximation}
For real $s$ and odd $n$, 
\bal{
H_n(s) &= \frac{ \Gamma\Big(\frac{n}{2} + 1\Big)}{\sqrt{\pi\big(n + \frac{1}{2}\big)}} 
2^{\frac{n + 1}{2}} (-1)^{\frac{n - 1}{2}} 
\\
&\times\Bigg(\sin\bigg(s\sqrt{n+\frac{1}{2}}\bigg) + 
\frac{s^3}{24\sqrt{n+\frac{1}{2}}}\cos\bigg(s\sqrt{n+\frac{1}{2}}\bigg)\Bigg) + O(n^{-1}).
}
For real $s$ and even $n$, 
\bal{
H_n(s) &= \frac{ \Gamma\Big(\frac{n+1}{2}\Big)}{\sqrt{\pi}} 
2^{\frac{n}{2}} (-1)^{\frac{n}{2}} 
\\
&\times\Bigg(\cos\bigg(s\sqrt{n+\frac{1}{2}}\bigg) + 
\frac{s^3}{24\sqrt{n+\frac{1}{2}}}\sin\bigg(s\sqrt{n+\frac{1}{2}}\bigg)\Bigg)+ O(n^{-1}).
}
The bound for the error term holds uniformly in every finite real interval. whether it contains the origin or not.
\end {theo}

\section{Proof of Theorem  \ref{propo_scalar_product}}
\label{sectionPropoScalar}

We start with a useful lemma.

\begin{lemma}
For every non-negative integers $l$ and $k,$ \\
the integral  $\int_{\Lambda}\left\vert z\right\vert ^{2l}z^{k}d\mu \left( z\right) $ is real.
\end{lemma}

\begin{proof} We write $z=x_{1}\mathbf{i}+x_{2}\mathbf{j}+x_{3}\mathbf{k}$
and expand the expression $\left\vert z\right\vert ^{2l}z^{k}.$ We claim
that any monomial before an imaginary unit has one of the variables $x_{i}$
in the odd power.

Indeed, it is sufficient to prove this claim for $l=0,$ since $\left\vert
z\right\vert ^{2l}$ is real and all monomials in its expansion have variables in the even power. 

Consider a single term in the expansion of $z^{k},$
for example, $x_{1}\mathbf{i}x_{3}\mathbf{k}x_{2}\mathbf{j}x_{2}\mathbf{%
j\ldots ,}$ It can be either imaginary or real, and it is clear that it is imaginary if and only if the term contain at least one of the
variables in the odd power. Indeed, we can do transpositions of imaginary units in the expansion and this will only introduce real factors. Hence, if all powers are even then all imaginary units in the product can be paired off and cancelled out, so that the product is real.  

Next, note that if any of the variables enters a monomial in the odd
power, then the integral of this monomial with respect to measure $\mu $ is $0,$ by the symmetry of $\mu .$ This concludes the proof of the lemma. 
\end{proof}

Now, let us calculate the real part of the expression $\left\vert
z\right\vert ^{2l}z^{k}.$ Note that $\left\vert z\right\vert ^{2l} = \left( x_{1}^{2}+x_{2}^{2}+x_{3}^{2}\right) ^{l}.$
Hence we only need to calculate the real part of $z^{k}.$ This is done in
the following lemma.

\begin{lemma}
Let $z=x_{1}\mathbf{i}+x_{2}\mathbf{j}+x_{3}\mathbf{k.}$ Then, 
\begin{equation*}
\mathrm{Re}z^{k}=\mathrm{Re}\overline{z}^{k}=\left\{ 
\begin{array}{cc}
\left( -1\right) ^{r}\left( x_{1}^{2}+x_{2}^{2}+x_{3}^{2}\right) ^{r}, & 
\text{if }k=2r, \\ 
0, & \text{if }k\text{ is odd.}%
\end{array}%
\right.
\end{equation*}
\end{lemma}

Remark: Since $|z^k|^2 = \left( x_{1}^{2}+x_{2}^{2}+x_{3}^{2}\right) ^{k}$,
this result implies that for an even $k$, we have $\mathrm{Re}z^{k}=|z^k|$ and
therefore $z^k$ is real. For an odd $k$, $z^k$ is imaginary. This is similar to
the situation for complex numbers.

\begin{proof} We write the quaternion $z$ in its matrix form:%
\begin{equation*}
\varphi (z\mathbf{)}=\left( 
\begin{array}{cc}
x_{1}i & x_{2}+x_{3}i \\ 
-x_{2}+x_{3}i & -x_{1}i%
\end{array}%
\right) ,
\end{equation*}%
and note that for every quaternion $w$ its real part can be computed as $%
\frac{1}{2}\mathrm{Tr}\varphi (w\mathbf{).}$ The eigenvalues of $\varphi (z%
\mathbf{)}$ are $\pm i\sqrt{x_{1}^{2}+x_{2}^{2}+x_{3}^{2}}.$ Hence, we
compute: 
\begin{equation*}
\mathrm{Re}z^{k}=\frac{1}{2}\mathrm{Tr}\varphi (z^{k})=\left\{ 
\begin{array}{cc}
\left( -1\right) ^{r}\left( x_{1}^{2}+x_{2}^{2}+x_{3}^{2}\right) ^{l}, & 
\text{if }k=2l, \\ 
0, & \text{if }k\text{ is odd.}%
\end{array}%
\right.
\end{equation*}%
The case of $\mathrm{Re}\overline{z}^{k}$ is similar.
\end{proof}

Now we can finish the proof of Theorem \ref{propo_scalar_product}.

Let $m\leq n$ and note that 
\begin{equation*}
\Re (\ovln{z}^m z^n)=|z|^{2m} \Re z^{n-m} = 
\begin{cases}
(-1)^{\frac{n-m}{2}}(x_1^2+x_2^2+x_3^2)^{\frac{m+n}{2}},
& \text{ if } n - m \text{ is even,} 
\\
0, & \text{ otherwise.}
\end{cases}
\end{equation*}

Next we calculate:%
\begin{eqnarray*}
\int_{\mathbb{R}^{3}}\left( x_{1}^{2}+x_{2}^{2}+x_{3}^{2}\right) ^{l}d\mu
(z) &=&\frac{1}{\left( 2\pi \right) ^{3/2}}\int_{\mathbb{S}^{2}}\int_{%
\mathbb{R}}r^{2l+2}e^{-r^{2}/2}drdS \\
&=&2\frac{1}{\sqrt{2\pi }}\int_{\mathbb{R}}r^{2l+2}e^{-r^{2}/2}dr \\
&=&\frac{2^{l+1}}{\sqrt{\pi }}\int_{\mathbb{R}}t^{l+1/2}e^{-t}dt \\
&=&\frac{2^{l+1}}{\sqrt{\pi }}\Gamma \left( l+\frac{3}{2}\right) .
\end{eqnarray*}

Hence, 
\begin{equation*}
\int \overline{z}^{m}z^{n}d\mu (z)=\left( -1\right) ^{\frac{n-m}{2}}\frac{2^{%
\frac{m+n}{2}+1}}{\sqrt{\pi }}\Gamma \left( \frac{m+n+3}{2}\right) =\left(
-1\right) ^{\frac{n-m}{2}}(m + n + 1)!!
\end{equation*}

\hfill $\square $

\bibliographystyle{plain}
\bibliography{comtest}

\end{document}